\documentclass{article}
\usepackage[top=30truemm,bottom=30truemm,left=25truemm,right=25truemm]{geometry}
\usepackage[dvipdfmx]{graphicx}
\usepackage{amsmath, amssymb}
\usepackage{amsthm}
\usepackage{autobreak}
\usepackage{amsfonts}
\usepackage{mathdots}
\usepackage{fancybox}
\usepackage{esint}
\usepackage{mathrsfs}
\usepackage{mathtools}
\usepackage{ascmac}
\usepackage{comment}
\usepackage{proof}
\usepackage{thm-restate}
\usepackage{setspace}
\theoremstyle{definition}
\newtheorem{theorem}{Theorem}
\newtheorem{definition}[theorem]{Definition}

\newtheorem{proposition}[theorem]{Proposition}

\usepackage{amscd}
\usepackage[all]{xy}
\usepackage{color}
\usepackage{cases}
\usepackage{here}
\usepackage{bm}
\usepackage{authblk}
\title{Compact embedding from variable-order Sobolev space to $L^{q(x)}(\Omega)$ and its application to Choquard equation with variable order and variable critical exponent
}
\author{Masaki Sakuma\thanks{Email: masakisakuma0110@gmail.com}}
\affil{Graduate School of Mathematical Sciences,\\ The University of Tokyo, Meguro-ku, Tokyo, Japan}

\begin{document}
\maketitle
\begin{abstract}
In this paper, we prove the compact embedding from the variable-order Sobolev space $W^{s(x,y),p(x,y)}_0 (\Omega)$ to the Nakano space $L^{q(x)}(\Omega)$ with a critical exponent $q(x)$ satisfying some conditions. It is noteworthy that the embedding can be compact even when $q(x)$ reaches the critical Sobolev exponent $p_s^*(x)$. 
As an application, we obtain a nontrivial solution of the Choquard equation 
\begin{equation*}
\displaystyle (-\Delta)_{p(\cdot,\cdot)}^{s(\cdot,\cdot)}u+|u|^{p(x,x)-2}u=\left(\int_{\Omega}\frac{|u(y)|^{r(y)}}{|x-y|^{\frac{\alpha(x)+\alpha(y)}{2}}}dy\right) |u(x)|^{r(x)-2}u(x)\quad\text{in $\Omega$}
\end{equation*}
with variable upper critical exponent in the sense of Hardy-Littlewood-Sobolev inequality under an appropriate boundary condition. 

\vspace{1ex}\par
\end{abstract}
\section{Introduction}
In the present paper, we prove the existence of a nontrivial solution of the variable-order $p(\cdot,\cdot)$-Choquard equation 
\begin{equation}\label{cho}
\left\{
\begin{array}{rl}
\displaystyle (-\Delta)_{p(\cdot,\cdot)}^{s(\cdot,\cdot)}u(x)+|u(x)|^{p(x,x)-2}u(x)=\left(\int_{\Omega}\frac{|u(y)|^{r(y)}}{|x-y|^{\frac{\alpha(x)+\alpha(y)}{2}}}dy\right) |u(x)|^{r(x)-2}u(x) &\text{in $\Omega$} \\
u=0 &\text{on $\partial\Omega$}
\end{array}
\right.
\end{equation}
with variable critical exponent. Here, ``critical'' means that $|u|^{r(x)}$ in the Choquard term can have the upper critical growth in the sense of the Hardy-Littlewood-Sobolev inequality. $\Omega\subset\mathbb{R}^N$ is a bounded Lipschitz domain. In the equation \eqref{cho}, $\alpha\in C(\overline{\Omega})$ satisfies $0<\alpha^{-}\coloneqq\inf\alpha\leq\alpha^{+}\coloneqq\sup\alpha<N$; $s\in C^{\mathrm{log}}(\overline{\Omega}\times \overline{\Omega})$ is a symmetric log-H\"{o}lder function that satisfies $0<s^{-}\coloneqq \inf s\leq s^{+}\coloneqq \sup s<1$; $p\in C^{\mathrm{log}}(\overline{\Omega}\times \overline{\Omega})$ is a symmetric log-H\"{o}lder function that satisfies $1<p^{-}\coloneqq \inf p\leq p^{+}\coloneqq \sup p<N/s^{+}$, and $r\in C(\overline{\Omega})$ satisfies
\[
\left(1-\frac{\alpha^{-}}{2N}\right)p(x,x)\leq r(x)\leq \left(1-\frac{\alpha^{+}}{2N}\right)p_s^*(x),
\]
where $\displaystyle p_s^*(x)\coloneqq \frac{Np(x,x)}{N-p(x,x)s(x,x)}$ is the critical Sobolev exponent. If
\[
\left\{x\in\Omega\, \middle|\, r(x)= \left(1-\frac{\alpha^{+}}{2N}\right)p_s^*(x)\right\}\neq\emptyset,
\]
then we say $|u|^{r(x)}$ has the upper critical growth in the sense of Hardy-Littlewood-Sobolev inequality. Similarly, if
\[
\left\{x\in\Omega\, \middle|\, r(x)= \left(1-\frac{\alpha^{-}}{2N}\right)p(x,x) \right\}\neq\emptyset,
\]
then we say $|u|^{r(x)}$ has the lower critical growth in the sense of Hardy-Littlewood-Sobolev inequality. \par
In the equation \eqref{cho}, the variable-order fractional $p(\cdot,\cdot)$-Laplace operator $(-\Delta)_{p(\cdot,\cdot)}^{s(\cdot,\cdot)}$ in $\Omega$ is given by the Cauchy principal value
\[
(-\Delta)_{p(\cdot,\cdot)}^{s(\cdot,\cdot)}u(x)\coloneqq \lim_{\varepsilon\to +0}\int_{\Omega\setminus B_\varepsilon (x)}\frac{|u(x)-u(y)|^{p(x,y)-2}(u(x)-u(y))}{|x-y|^{N+s(x,y)p(x,y)}}dy
\]
for $u\in C_c^\infty(\Omega)$ and its extension to the function space $W_{0}^{s(x,y),p(x,y)}(\Omega)$ defined in the next section. This operator is a generalization of the regional fractional Laplacian. We adopt this definition consistently from the next section onward. On the other hand, there is another definition of variable-order fractional $p(\cdot,\cdot)$-Laplacian:
\[
(-\Delta)_{p(\cdot,\cdot),\mathrm{res}}^{s(\cdot,\cdot)}u(x)\coloneqq \lim_{\varepsilon\to +0}\int_{\mathbb{R}^N\setminus B_\varepsilon (x)}\frac{|u(x)-u(y)|^{p(x,y)-2}(u(x)-u(y))}{|x-y|^{N+s(x,y)p(x,y)}}dy
\]
for the extension of any $u\in C_c^\infty(\Omega)$ by zero to $\mathbb{R}^N$, which is a generalization of the restricted fractional Laplacian and is adopted in \cite{variable-order_Xiang,variable-order_Cheng,Gan}. In such a case, to deal with an extension of the zero Dirichlet boundary condition, the boundary condition should be described as the exterior one $u=0$ in $\mathbb{R}^N\setminus\Omega$, not $u=0$ on $\partial\Omega$. However, this definition requires that $s(\cdot,\cdot)$ and $p(\cdot,\cdot)$ are defined in the whole space $\mathbb{R}^N$ and hence is inconsistent with our setting. In the case of constant exponents, some literature define
\[
(-\Delta)_{p}^{s}u(x)\coloneqq C_{N,s,p}\lim_{\varepsilon\to +0}\int_{\mathbb{R}^N\setminus B_\varepsilon (x)}\frac{|u(x)-u(y)|^{p-2}(u(x)-u(y))}{|x-y|^{N+ps}}dy
\]
with a scaling factor
\[
C_{N,s,p}\coloneqq\frac{2^{2s-2}ps(1-s)}{\pi^{\frac{N-1}{2}}}\cdot \frac{\Gamma(\frac{N+ps}{2})}{\Gamma(\frac{p+1}{2})\Gamma(2-s)}
\]
mainly to make it consistent with the normal $p$-Laplacian as $s\to 1^{-}$. Meanwhile, in the case of variable exponents, the scaling factor no longer makes sense, so we omit it. In the non-fractional case where $s\equiv 1$, the operator $(-\Delta)_{p(\cdot,\cdot)}^{s(\cdot,\cdot)}$ corresponds to $p(x)$-Laplacian for $p(x)=p(x,x)$, which is defined by
\[
(-\Delta)_{p(x)}u\coloneqq -\operatorname{div}(|\nabla u|^{p(x)-2}\nabla u).
\]

Throughout this paper, $B_r(x)$ denotes the open ball with radius $r$ centered at $x$ in $\mathbb{R}^N$; $C$, $C'$, $C_i$ and $C_i'$ represent various positive constants; $\|\cdot\|_q$ denotes the $L^q$ norm; $\chi_A$ denotes the indicator function of $A$; $C_c^\infty(\Omega)$ denotes the set of all $C^\infty$ functions compactly supported in $\Omega$; $C^\mathrm{log}(\Omega)$ denotes the set of all log-H\"{o}lder continuous functions defined in $\Omega$. The definition of log-H\"{o}lder continuity will be described in the section of embeddings later. We adopt the notation $\Phi^{+}\coloneqq\sup\Phi$ and $\Phi^{-}\coloneqq\inf\Phi$ for any variable exponent $\Phi$. In addition, we introduce the quantity $\displaystyle\sigma_\alpha(x)\coloneqq \frac{2N}{2N-\alpha(x)}$ depending on $\alpha(x)$. \par

The main result in this paper is as follows:
\begin{theorem}\label{CriticalCompactMountainPass}
Assume the continuous embedding $W^{s(x,y),p(x,y)}(\Omega)\hookrightarrow L^{p_s^*(x)}(\Omega)$ holds. Moreover, assume the following condition on touching rate of $r(x)$ to the upper critical exponent: 
\begin{itemize}
\item[(TR)] There exist $x_0\in\Omega$, $\beta\in (0,1)$, $C_0>0$ and $\eta>0$ such that for any $\rho\in (0,\eta)$, we have
\[
\max_{x\in \partial B_\rho (x_0)}r(x)\leq \min_{x\in \partial B_\rho (x_0)} \left(1-\frac{\alpha^{+}}{2N}\right) p_s^*(x)-\frac{C_0}{(-\log\rho)^\beta}.
\]
\end{itemize}
Then there exists a nontrivial weak solution $u\in W_{0}^{s(x,y),p(x,y)}(\Omega)$ to \eqref{cho}. 
\end{theorem}
At the same time, we discuss a sufficient condition to ensure the critical continuous Sobolev-type embedding
\[
W^{s(x,y),p(x,y)}(\Omega)\hookrightarrow L^{p_s^*(x)}(\Omega).
\]
\begin{restatable}{theorem}{CriticalContinuousEmbedding} \label{CriticalContinuousEmbedding}
Let $\Omega\subset\mathbb{R}^N$ be a bounded Lipschitz domain. Assume that symmetric and log-H\"{o}lder continuous functions $p:\overline{\Omega}^2\to\mathbb{R}$ and $s: \overline{\Omega}^2\to\mathbb{R}$ satisfy $0<s^{-}\leq s^{+}<1$ and $1<p^{-}\leq p^{+}<N/s^{+}$. In addition, assume that any point of the diagonal set $\{(x,y)\in \overline{\Omega}^2 \mid x=y\}$ is a (not necessarily strict) local minimum point of $s(x,y)$ and $p(x,y)$. Let $q:\overline{\Omega}\to\mathbb{R}$ be a continuous function such that $p^{+}\leq q^{-}\leq q(x)\leq p_s^*(x)$. Then, the continuous embedding $W^{s(x,y),p(x,y)}(\Omega)\hookrightarrow L^{q(x)}(\Omega)$ holds.
\end{restatable}
This extends the result of the critical embedding in \cite{s_variable_exponent_CC} to the case where $s(x,y)$ is variable. Note that the uniform continuity of exponents does not assure the critical continuous embedding and we need some additional condition such as the log-H\"{o}lder continuity of exponents even if $s\equiv 1$ (see Corollary 8.3.2 and Proposition 8.3.7 in \cite{Diening}). \par
Furthermore, we also found a sufficient condition for the embedding $W^{s(x,y),p(x,y)}(\Omega)\hookrightarrow L^{q(x)}(\Omega)$ to be compact even though $q(x)$ reaches the critical Sobolev exponent. 
\begin{restatable}{theorem}{CriticalCompactEmbedding} \label{CriticalCompactEmbedding}
Let $\Omega\subset\mathbb{R}^N$ be a bounded Lipschitz domain. Let $s,p\in C^{\mathrm{log}}(\overline{\Omega}\times\overline{\Omega})$ be symmetric log-H\"{o}lder functions satisfying $0<s^{-}\leq s^{+}<1<p^{-}\leq p^{+}<N/s^{+}$ and $q\in C(\overline{\Omega})$ be a continuous function such that $p^{+}\leq q^{-}\leq q(x)\leq p_s^*(x)$ ($\forall x\in\Omega$). Assume the continuous embedding $W^{s(x,y),p(x,y)}(\Omega)\hookrightarrow L^{p_s^*(x)}(\Omega)$ holds. 
Suppose that there exist $x_0\in\Omega$, $\beta\in (0,1)$, $C_0>0$ and $\eta>0$ such that for any $\rho\in (0,\eta)$, we have
\begin{equation}\label{CriticalRate}
\max_{x\in \partial B_\rho (x_0)}q(x)\leq \min_{x\in \partial B_\rho (x_0)}p_s^*(x)-\frac{C_0}{(-\log\rho)^\beta}
\end{equation}
and
\[
\inf_{x\in\Omega\setminus B_\eta (x_0)}(p_s^*(x)-q(x))>0
\]
Then, the embedding $W^{s(x,y),p(x,y)}(\Omega)\hookrightarrow L^{q(x)}(\Omega)$ is compact. 
\end{restatable}
This extends the compactness result in \cite{KurataShioji} to general variable-exponent case including Sobolev spaces with variable smoothness. This is also almost a necessary condition in the following sense.
\begin{restatable}{theorem}{CriticalNotCompactEmbedding}\label{CriticalNotCompactEmbedding}
Let $\Omega\subset\mathbb{R}^N$ be a bounded domain. Let $q:\Omega\to\mathbb{R}$ be a measurable function and $s,p\in C^\gamma (\overline{\Omega}\times \overline{\Omega})$ be symmetric $\gamma$-H\"{o}lder continuous functions such that $0<s^{-}\leq s^{+}<1<p^{-}\leq p^{+}< N/s^{+}$, where $0<\gamma<1$. Assume there exists a point $x_0\in\Omega$ and constants $C_0>0$, $\eta>0$ such that for almost every $x\in\Omega\cap B_\eta(x_0)$, we have
\[
q(x)\geq p_s^*(x)-\frac{C_0}{\log(1/|x-x_0|)}.
\]
Then, even if the embedding from $W^{s(x,y),p(x,y)}(\Omega)$ to $L^{q(x)}(\Omega)$ holds, it is not compact.
\end{restatable}
Theorem \ref{CriticalCompactEmbedding} enables us to recover the loss of compactness of (PS) sequences without the concentration compactness method. The proof of Theorem \ref{CriticalCompactMountainPass} is done using this compact embedding to assure the Palais–Smale condition of the energy functional corresponding to \eqref{cho}. \par
The variable-order fractional Laplacian and nonlinear elliptic equations involving variable integrability exponents are subjects of special interest in the recent study on partial differential equations, along with real-world applications such as electrotheological fluids \cite{fluid} and image restoration \cite{image}. Bahrouni and R\u{a}dulescu \cite{Bahrouni} introduced a fractional $p(\cdot,\cdot)$-Laplacian
\begin{align*}
(-\Delta)_{p(\cdot,\cdot)}^{s}u(x)&=\mathrm{p.v.}\int_{\Omega}\frac{|u(x)-u(y)|^{p(x,y)-2}(u(x)-u(y))}{|x-y|^{N+sp(x,y)}}dy \\
&= \lim_{\varepsilon\to +0}\int_{\Omega\setminus B_\varepsilon (x)}\frac{|u(x)-u(y)|^{p(x,y)-2}(u(x)-u(y))}{|x-y|^{N+sp(x,y)}}dy,
\end{align*}
where $s$ is a constant, and proved the existence of solutions to the Dirichlet boundary value problem
\[
\left\{
\begin{array}{rl}
(-\Delta)_{p(\cdot,\cdot)}^{s}u(x)+|u(x)|^{q_1(x)-2}u(x)=\lambda |u(x)|^{q_2(x)-2}u(x) &\text{in $\Omega$}\\
u=0& \text{on $\partial\Omega$}
\end{array}
\right.
\]
using the Ekeland variational principle. Cheng et al. \cite{variable-order_Cheng} considered the variable-order fractional $p(x,y)$-Laplacian as an extension of the restricted fractional Laplacian and obtained the existence of weak solutions of nonlinear elliptic equations with the exterior boundary condition of the form
\[
\left\{
\begin{array}{rl}
(-\Delta)_{p(\cdot,\cdot),\mathrm{res}}^{s(\cdot,\cdot)}u(x)+|u(x)|^{p(x)-2}u(x)=f(x)h(u) &\text{in $\Omega$}\\
u=0& \text{in $\mathbb{R}^N\setminus\Omega$}.
\end{array}
\right.
\]
Inspired by this, Liu and Fu \cite{variable-order_p(x)_Sobolev} further studied the existence of solutions to the boundary value problem
\[
\left\{
\begin{array}{rl}
(-\Delta)_{p(\cdot,\cdot)}^{s(\cdot,\cdot)}u(x)+V(x) |u(x)|^{p(x)-2}u(x)=f(x,u)+g(x) &\text{in $\Omega$}\\
u=0& \text{on $\partial\Omega$}
\end{array}
\right.
\]
with the variable-order fractional $p(x,y)$-Laplacian as an extension of the regional fractional Laplacian under certain conditions. \par
In the case of constant exponents and non-fractional Laplacian, the equation \eqref{cho} goes back to the Choquard equation
\begin{equation}\label{Pekar}
-\Delta u+u=\left(\frac{1}{|x|^\alpha}\ast |u|^q\right)|u|^{q-2}u\quad\text{in $\Omega$},
\end{equation}
which was first proposed as a mathematical model for polaron by Pekar \cite{Pekar} in 1954. The classical results for this equation and its variants are summarized in \cite{Moroz}. In recent years, its extensions to fractional order and variable exponents have attracted much attention. For the case of constant smoothness, especially $s=1$, there are more and more studies on critical Choquard-type equations with variable exponents, such as \cite{Zhang,Fu_Zhang,critical_fractional_p(x),critical_p(x)_Kirchhoff}. However, there are still relatively few studies on variable-order critical Choquard-type equations. One reason for this is that results for critical continuous embedding are still inadequate. \par

Note that in the case with the constant critical exponent, if $\Omega$ is a star-shaped domain, the equation \eqref{cho} does not have nontrivial solutions in general. For instance, concerning a non-existence result due to the Pohozaev identity in the case $p=2$, see Theorem 2.8 in \cite{FractionalPohozaev}. This suggests that the existence of a nontrivial solution of \eqref{cho} is attributed to the variability of the exponents. \par
The remainder of this paper is structured as follows. The next preliminary section provides basic knowledge about function paces with variable exponents. In Section 3, we restate and prove some results about the variable-order embeddings. Section 4 is devoted to giving a proof of our main theorem. 

\section{Preliminary}
Let $\Omega\subset\mathbb{R}^N$ be a domain. We define
\[
C_{+}(\bar{\Omega})\coloneqq \{\Phi\in C(\bar{\Omega};\mathbb{R})\mid 1<\Phi^{-}\leq \Phi^{+}<\infty\}.
\]
For $p\in C_{+}(\bar{\Omega})$, the Lebesgue space with a variable exponent $p(x)$ is defined by
\[
L^{p(\cdot )}(\Omega)=L^{p(x)}(\Omega)\coloneqq\left\{u:\Omega\to\mathbb{R}\;\middle|\;\text{$u$ is measurable and }\int_\Omega |u(x)|^{p(x)}dx<\infty\right\}
\]
endowed with the Luxemburg norm
\[
\|u\|_{p(x)}\coloneqq\inf\left\{\lambda>0\;\middle|\; \int_\Omega \left|\frac{u(x)}{\lambda}\right|^{p(x)}dx\leq 1\right\}.
\]
A Lebesgue space with a variable exponent is also called \textit{a Nakano space}. This space is a separable and uniformly convex Banach space. We consider the modular map $\varrho_p: L^{p(x)}(\Omega)\to\mathbb{R}$ defined by
\[
\varrho_{p(x)}(u)\coloneqq \int_{\Omega}|u(x)|^{p(x)}dx.
\]
There is an important relation between the modular and the corresponding norm. 
\begin{proposition}
For $u\in L^{p(x)}(\Omega)$, we have
\begin{enumerate}
\item[(i)] $\|u\|_{p(x)}>1\;\text{[resp. $=1$, $<1$]}\Leftrightarrow \varrho_{p}>1\; \text{[resp. $=1$, $<1$]}$;
\item[(ii)] $\|u\|_{p(x)}\geq 1\Rightarrow \|u\|_{p(x)}^{p^{-}}\leq \varrho_{p}(u)\leq \|u\|_{p(x)}^{p^{+}}$;
\item[(iii)] $\|u\|_{p(x)}\leq 1\Rightarrow \|u\|_{p(x)}^{p^{+}}\leq \varrho_{p}(u)\leq \|u\|_{p(x)}^{p^{-}}$. 
\end{enumerate}
In particular, the convergence with respect to the modular is equivalent to the norm convergence in $L^{p(x)}(\Omega)$, i.e., for any $\{u_n\}\subset L^{p(x)}(\Omega)$, we have
\[
\|u_n\|_{p(x)}\to 0\Longleftrightarrow \varrho_{p}(u_n)\to 0\quad\text{and}\quad \|u_n\|_{p(x)}\to \infty\Longleftrightarrow \varrho_{p}(u_n)\to \infty.
\]
\end{proposition}
To handle two or more different variable exponents at the same time, we use the following inequalities.
\begin{proposition}
Let $p\in C_{+}(\bar{\Omega})$ and $q\in C(\bar{\Omega})$. Assume $pq\in C_{+} (\bar{\Omega})$ and $u\in L^{p(x)q(x)}(\Omega)$. Then, 
\begin{enumerate}
\item $\|u\|_{p(x)q(x)}\geq 1\Rightarrow \|u\|_{p(x)q(x)}^{q^{-}}\leq \| |u|^q\|_{p(x)}\leq \|u\|_{p(x)q(x)}^{q^{+}}$;
\item $\|u\|_{p(x)q(x)}\leq 1\Rightarrow \|u\|_{p(x)q(x)}^{q^{+}}\leq \| |u|^q\|_{p(x)}\leq \|u\|_{p(x)q(x)}^{q^{-}}$.
\end{enumerate}
\end{proposition}
Let $p'(x)$ denote the H\"{o}lder conjugate exponent of $p(x)$, i.e., $1/p(x)+1/p'(x)=1$. Then, the following H\"{o}lder type inequality holds for variable-exponent spaces. 
\begin{proposition}
Let $p\in C_{+}(\bar{\Omega})$. For any $u\in L^{p(x)}(\Omega)$ and $v\in L^{p'(x)}(\Omega)$, we have
\[
\left|\int_\Omega uv dx\right|\leq \left(\frac{1}{p^{-}}+ \frac{1}{(p')^{-}}\right)\|u\|_{p(x)}\|v\|_{p'(x)}\leq 2 \|u\|_{p(x)}\|v\|_{p'(x)}.
\]
\end{proposition}
From this, we deduce that if $\displaystyle \frac{1}{q(x)}=\frac{1}{p_1(x)}+ \frac{1}{p_2(x)}$, then for any $u_1\in L^{p_1(x)}$ and $u_2\in L^{p_2(x)}$, then we have
\[
\|u_1 u_2\|_{q(x)}\leq\left(\left(\frac{q}{p_1}\right)^{+}+ \left(\frac{q}{p_2}\right)^{+}\right)\|u_1\|_{p_1(x)}\|u_2\|_{p_2(x)}.
\]
For details and proof, see, e.g., \cite{Diening,Fan_Lp(x)}. \par

Next, we introduce variable-order Sobolev spaces. For symmetric continuous functions $p:\bar{\Omega}\times \bar{\Omega} \to (1,\infty)$ and $s: \bar{\Omega}\times \bar{\Omega} \to (0,1)$ such that $0<s^{-}\leq s^{+}<1$ and $s^{+}p^{+}<N$, we define the modular
\[
\varrho(u)= \varrho_{W^{s(x,y),p(x,y)}(\Omega)}(u)\coloneqq \int_\Omega\int_\Omega\frac{|u(x)-u(y)|^{p(x,y)}}{|x-y|^{N+s(x,y)p(x,y)}}dxdy+ \int_\Omega |u|^{p(x,x)}dx
\]
and the corresponding Luxemburg norm
\[
\|u\|=\|u\|_{s(x,y),p(x,y);\Omega}\coloneqq\inf\left\{\lambda>0\;\middle|\;\varrho\left(\frac{u}{\lambda}\right)=\int_\Omega\int_\Omega\frac{|u(x)-u(y)|^{p(x,y)}}{\lambda^{p(x,y)} |x-y|^{N+s(x,y)p(x,y)}}dxdy+ \int_\Omega \frac{|u|^{p(x,x)}}{\lambda^{p(x,x)}}dx <1\right\}
\]
letting $\inf\emptyset=\infty$. The Sobolev space $W^{s(x,y),p(x,y)}(\Omega)$ with variable exponents $s(x,y)$ and $p(x,y)$ is defined by
\[
W^{s(x,y),p(x,y)}(\Omega)\coloneqq\left\{u\in L^{p(x,x)}(\Omega)\mid \|u\|_{s(x,y),p(x,y);\Omega} <\infty\right\}.
\]
Furthermore, the closure of $C_c^\infty(\Omega)$ with respect to $\|\cdot\|_{s(x,y),p(x,y);\Omega}$ is denoted by $W_{0}^{s(x,y),p(x,y)}(\Omega) \subset W^{s(x,y),p(x,y)}(\Omega)$. These spaces are separable and uniformly convex. This definition of $W_{0}^{s(x,y),p(x,y)}(\Omega)$ is adopted in \cite{variable-order_p(x)_Sobolev,Nguyen}. However, note that there is another slightly different definition: \cite{variable-order_p(x),variable-order_H,Chammem} adopt another one
\[
\tilde{W}_{0}^{s(x,y),p(x,y)}(\Omega)=\{u\in W^{s(x,y),p(x,y)}(\mathbb{R}^N)\mid u=0\text{ a.e. in $\mathbb{R}^N\setminus\Omega$}\}.
\]
Even if we identify any function defined on $\Omega$ with its zero extension, these definitions are not equivalent, and the latter defines a smaller space with a stronger topology. In addition, the latter definition requires that $p(\cdot,\cdot)$ and $s(\cdot,\cdot)$ are defined in the whole space, so it does not fit our situation. In this paper, we adopt the former definition based on the closure. On the other hand, if one considers $\tilde{W}_{0}^{s(x,y),p(x,y)}(\Omega)$ as a space to which the weak solutions belong, it is natural to treat the generalized restricted fractional Laplacian $(-\Delta)_{p(\cdot,\cdot),\mathrm{res}}^{s(\cdot,\cdot)}$ instead of $(-\Delta)_{p(\cdot,\cdot)}^{s(\cdot,\cdot)}$ under the exterior boundary condition $u=0$ a.e. in $\mathbb{R}^N\setminus\Omega$, which is also natural in the constant exponent fractional case. \par

Here, we state a variable-exponent version of the Hardy-Littlewood-Sobolev inequality (see \cite{variable_exponent_HLS}). We use this inequality to estimate the energy corresponding to the Choquard term in \eqref{cho}. 
\begin{proposition}\label{HLS}
Let $p_1,p_2\in C^{+}(\Omega)$ and $\alpha:\overline{\Omega}\times \overline{\Omega}\to\mathbb{R}$ be a continuous function such that $0<\alpha^{-}\leq\alpha^{+}<N$. If 
\[
\frac{1}{p_1(x)}+ \frac{\alpha(x,y)}{N}+\frac{1}{p_2(y)}=2\quad (\forall x,y\in\Omega),
\]
then there exists a constant $C>0$ such that for any $f_1\in L^{p_1^{+}} (\Omega)\cap L^{p_1^{-}} (\Omega)$ and for any $f_2\in L^{p_2^{+}} (\Omega)\cap L^{p_2^{-}} (\Omega)$, we have
\[
\left|\int_{\Omega} \int_{\Omega}\frac{f_1(x)f_2(y)}{|x-y|^{\alpha(x,y)}}dxdy\right|\leq C(\|f_1\|_{L^{p_1^{+}}} \|f_2\|_{L^{p_2^{+}}}+ \|f_1\|_{L^{p_1^{-}}} \|f_2\|_{L^{p_2^{-}}}).
\]
\end{proposition}
In the setting of \eqref{cho}, noting that
\[
\frac{1}{\sigma_\alpha(x)}+\frac{\alpha(x)+\alpha(y)}{2N}+\frac{1}{\sigma_\alpha(y)}=2
\]
and that
\[
p(x,x)\leq r(x)\sigma_\alpha^{-}\leq r(x)\sigma_\alpha^{+}\leq p_s^*(x),
\]
by applying Proposition \ref{HLS} to the case where $p_1=p_2=\sigma_\alpha$ and $f_1=f_2=|u|^{r(x)}$ for any $u\in W^{s(x,y),p(x,y)}(\Omega)\hookrightarrow L^{r(x)\sigma_\alpha^{-}} (\Omega)\cap L^{r(x)\sigma_\alpha^{+}} (\Omega)$, we obtain
\begin{align*}
&\phantom{=}\int_{\Omega} \int_{\Omega}\frac{|u(x)|^{r(x)}|u(y)|^{r(y)}}{|x-y|^{\alpha(x,y)}}dxdy\\ 
&\leq C\max\{\|u\|_{r(x)\sigma_\alpha^{-}}^{2 r^{+}}, \|u\|_{r(x)\sigma_\alpha^{-}}^{2 r^{-}}, \|u\|_{r(x)\sigma_\alpha^{+}}^{2 r^{+}}, \|u\|_{r(x)\sigma_\alpha^{+}}^{2 r^{-}}\}.
\end{align*}
Furthermore, inferring as in \cite{variable_exponent_HLS}, we can see the functional
\[
K: W^{s(x,y),p(x,y)}(\Omega)\to\mathbb{R}; u\mapsto \int_{\Omega} \int_{\Omega}\frac{|u(x)|^{r(x)}|u(y)|^{r(y)}}{2r(x)|x-y|^{\alpha(x,y)}}dxdy
\]
is Fr\'{e}chet differentiable with derivative
\[
K'[u]v=\int_{\Omega} \int_{\Omega}\frac{|u(x)|^{r(x)}|u(y)|^{r(y)-2}v(y)}{|x-y|^{\alpha(x,y)}}dxdy.
\]

\section{Embeddings}
In this section, we establish some variable-order versions of the Sobolev embedding theorem with variable exponents. 
As for the subcritical case, combining the proofs of Theorem 3.2 in \cite{variable-order_p(x)_Sobolev} and Theorem 3.2 in \cite{uniformly_continuous} (or Theorem 3.6 in \cite{variable-order_p(x)}), we obtain the following compact embedding result.
\begin{proposition}
Let $\Omega\subset\mathbb{R}^N$ be a bounded Lipschitz domain. Assume that continuous and symmetric functions $p\in C_{+}(\bar{\Omega}\times \bar{\Omega})$ and $s: \bar{\Omega}\times \bar{\Omega} \to\mathbb{R}$ satisfy $0<s^{-}\leq s^{+}<1$ and $s^{+}p^{+}<N$. Let $q\in C_{+}(\bar{\Omega})$ be such that $p(x,x)\leq q(x)< p_s^*(x)$ for any $x\in \bar{\Omega}$ (and thus $\inf(p_s^*-q)>0$ since $\bar{\Omega}$ is compact). Then, there exists a constant $C>0$ such that for any $u\in W^{s(x,y),p(x,y)}(\Omega)$, we have
\[
\|u\|_{L^{q(x)}(\Omega)}\leq C \|u\|_{W^{s(x,y),p(x,y)}(\Omega)}.
\]
Moreover, this embedding is compact.
\end{proposition}
However, as for the critical embedding, the situation is more complicated. As emphasized in \cite{Diening}, a concept called log-H\"{o}lder continuity often plays an important role in considering variable exponents. In particular, it is necessary to assume the log-H\"{o}lder continuity of exponents when considering the critical embedding. 
\begin{definition}
Let $f:\Omega\to\mathbb{R}$ be a continuous function defined in a bounded domain $\Omega\subset\mathbb{R}^N$. If there exists $C>0$ such that for any $x,y\in\Omega$, we have
\[
|f(x)-f(y)|\log\frac{1}{|x-y|}\leq C,
\]
then we say $f$ is (locally) log-H\"{o}lder continuous. Let $C^\mathrm{log}(\Omega)$ stand for the set of all log-H\"{o}lder continuous functions defined in $\Omega$. 
\end{definition}
Under the assumption of the log-H\"{o}lder continuity, we now prove a new critical embedding theorem. 

\CriticalContinuousEmbedding*

\begin{proof}
Since any point $(x,x)\in\ell\coloneqq \{(x,y)\in\overline{\Omega}^2\mid x=y\}$ is a local minimum point of $s$ and $p$ and since $\ell$ is compact, there exists $\varepsilon\in (0,1)$ such that for any $x,y\in\overline{\Omega}$, 
\[
|x-y|<\varepsilon\Rightarrow p(x,y)\geq \underline{p}
\]
and
\[
|x-y|<\varepsilon\Rightarrow s(x,y)\geq\underline{s},
\]
where $\underline{p}$ [resp. $\underline{s}$] is the common local minimum of $p(\cdot,\cdot)$ [resp. $s(\cdot,\cdot)$], whose existence is assured by the fact that $\Omega$ is connected. Combining this with the log-H\"{o}lder continuity of $s$ and $p$, we obtain
\[
c_{p}\coloneqq\sup_{\substack{x,y\in\Omega \\ 0<|x-y|<1/2}}|p(x,y)-\underline{p}|\log\frac{1}{|x-y|}<\infty
\]
and
\[
c_{s}\coloneqq\sup_{\substack{x,y\in\Omega \\ 0<|x-y|<1/2}}|s(x,y)-\underline{s}| \log\frac{1}{|x-y|} <\infty.
\]

First, we impose the following additional assumption on $q(x)$:
\begin{itemize}
\item[(E1)] There exists $\varepsilon>0$ such that for any $x\in\overline{\Omega}$, we have
\[
q|_{B_\varepsilon(x)}^{+}\leq\frac{Np|_{B_\varepsilon(x)\times B_\varepsilon(x)}^{-}}{N-p|_{B_\varepsilon(x) \times B_\varepsilon(x)}^{-} s|_{B_\varepsilon(x) \times B_\varepsilon(x)}^{-}}.
\]
\end{itemize}

Cover $\Omega$ with a family $\{\tilde{Q}_i\}_{i\in\mathbb{N}}$ of countably many closed cubes with sides of length $\varepsilon>0$. Let $Q_i\coloneqq \tilde{Q}_i\cap \Omega$. 
By (E1) and the uniform continuity, taking $\varepsilon>0$ small enough, we have
\[
q_i^{+}\leq\frac{Np_i^{-}}{N-p_i^{-}s_i^{-}},
\]
where $q_i^{\pm}\coloneqq q|_{Q_i}^{\pm}$, $p_i^{\pm}\coloneqq p|_{(Q_i)^2}^{\pm}$, $s_i^{\pm}\coloneqq s|_{(Q_i)^2}^{\pm}$. Let $v\in W^{s(x,y),p(x,y)}(\Omega)$ and $i\in\mathbb{N}$. Define a measure $\mu_i$ and a function $F$ as follows:
\[
d\mu(x,y)\coloneqq\frac{dxdy}{|x-y|^{N-s_i^{-}p_i^{-}}},\quad F(x,y)\coloneqq \frac{|v(x)-v(y)|}{|x-y|^{2s_i^{-}}}.
\]
Consider the seminorm
\[
[u]_{s(x,y),p(x,y);O}\coloneqq\inf\left\{\lambda>0\,\middle|\, \int_O\int_O\frac{|u(x)-u(y)|^{p(x,y)}}{\lambda^{p(x,y)} |x-y|^{N+s(x,y)p(x,y)}}dxdy <1\right\}.
\]
Since
\begin{align*}
&\phantom{\coloneqq} \sup_{(x,y):x\neq y}|x-y|^{-(s(x,y)-s_i^{-})p(x,y)-s_i^{-}(p(x,y)-p_i^{-})} \\
&= \sup_{(x,y):x\neq y} e^{-(s(x,y)-s_i^{-})p(x,y)\log|x-y|-s_i^{-}(p(x,y)-p_i^{-})\log|x-y|} \\
&\leq \sup_{(x,y):x\neq y} e^{-p^{+}(s(x,y)-s_i^{-})\log|x-y|-s^{+}(p(x,y)-p_i^{-})\log|x-y|} \eqqcolon C_L<\infty,
\end{align*}
we have
\begin{align*}
1&=\int_{Q_i}\int_{Q_i}\frac{|v(x)-v(y)|^{p(x,y)}}{[v]_{s,p;Q_i}^{p(x,y)}|x-y|^{N+s(x,y)p(x,y)}}dxdy \\
&= \int_{Q_i}\int_{Q_i}\left|\frac{F(x,y)}{[v]_{s,p;Q_i}}\right|^{p(x,y)}\frac{1}{|x-y|^{-(s(x,y)-s_i^{-})p(x,y)-s_i^{-}(p(x,y)-p_i^{-})}} d\mu_i(x,y) \\
&\geq (C_L+1)^{-1} \int_{Q_i}\int_{Q_i}\left|\frac{F(x,y)}{[v]_{s,p;Q_i}}\right|^{p(x,y)}d\mu_i(x,y) \\
&\geq \int_{Q_i}\int_{Q_i}\left|\frac{F(x,y)}{(C_L+1)^{1/p_i^{-}} [v]_{s,p;Q_i}}\right|^{p(x,y)}d\mu_i(x,y).
\end{align*}
Therefore, $\|F\|_{L_{\mu_i}^{p(x,y)}(Q_i\times Q_i)}\leq (C_L+1)^{1/p_i^{-}} [v]_{s,p;Q_i}$. On the other hand,
\begin{align*}
\|F\|_{L_{\mu_i}^{p_i^{-}}(Q_i\times Q_i)}&= \left( \int_{Q_i}\int_{Q_i}\left|\frac{|v(x)-v(y)|}{|x-y|^{2s_i^{-}}}\right|^{p_i^{-}}\frac{dxdy}{|x-y|^{N-s_i^{-}p_i^{-}}} \right)^{1/p_i^{-}} \\
&= \left( \int_{Q_i}\int_{Q_i}\frac{|v(x)-v(y)|^{p_i^{-}}}{|x-y|^{N+s_i^{-}p_i^{-}}} dxdy \right)^{1/p_i^{-}}=[v]_{s_i^{-},p_i^{-};Q_i}
\end{align*}
and since $p_i^{-}\leq p|_{Q_i\times Q_i}$, by Corollary 3.3.4 in \cite{Diening},
\begin{align*}
\|F\|_{L_{\mu_i}^{p_i^{-}}(Q_i\times Q_i)}&\leq 2(1+\mu_i (Q_i\times Q_i)) \|F\|_{L_{\mu_i}^{p(x,y)}(Q_i\times Q_i)} \\
&\leq  2\left(1+|Q_i|\cdot\frac{N|B_1|\cdot 2^{s_i^{-}p_i^{-}}}{s_i^{-}p_i^{-}}\right) \|F\|_{L_{\mu_i}^{p(x,y)}(Q_i\times Q_i)}.
\end{align*}
Combining these, we obtain $[v]_{s_i^{-},p_i^{-};Q_i}\leq C [v]_{s(x,y),p(x,y);Q_i}$. Since $p_i^{-}\leq p|_{(Q_i)^2}$, we also have $\|v\|_{L^{p_i^{-}}(Q_i)}\leq 2(1+|Q_i|) \|v\|_{L^{p(x,x)}(Q_i)}$. By summing them, we get $\|v\|_{W^{s_i^{-},p_i^{-}}(Q_i)}\leq C\|v\|_{W^{s(x,y),p(x,y)}(Q_i)}$. \par
As in the proof of Theorem 3.5 in \cite{uniformly_continuous}, we can deduce the existence of an extension $\tilde{v}\in W^{s_i^{-},p_i^{-}}(\mathbb{R}^N)$ with compact support such that $\tilde{v}=v$ in $Q_i$, and $\|v\|_{L^{q_i^{+}}(Q_i)}\leq\| \tilde{v}\|_{L^{(p_i^{-})_{s_i^{-}}^*}(\mathbb{R}^N)}\leq C\|v\|_{W^{s_i^{-},p_i^{-}}(Q_i)}$. On the other hand, $\|v\|_{L^{q(x)}(Q_i)}\leq 2(1+|Q_i|) \|v\|_{L^{q_i^{+}}(Q_i)}$. Combining these, we obtain $\|v\|_{L^{q(x)}(Q_i)}\leq C \|v\|_{W^{s(x,y),p(x,y)}(Q_i)}$. \par
In order to prove the embedding, by the closed graph theorem and the homogeneity of the inequality, it suffices to show that we have $\varrho_{L^{q(x)}(\Omega)}(v)<\infty$ for any $v\in W^{s(x,y),p(x,y)}(\Omega)$ such that $\|v\|_{W^{s(x,y),p(x,y)}(\Omega)}=1$. Take such $v$ arbitrarily. Then, $\varrho_{W^{s(x,y),p(x,y)}(Q_i)}(v)\leq \varrho_{W^{s(x,y),p(x,y)}(\Omega)}(v)=1$ and thus $\|v\|_{W^{s(x,y),p(x,y)}(Q_i)}\leq 1$. 
\begin{align*}
\int_{Q_i}|v|^{q(x)}dx&\leq \max\{\|v\|_{L^{q(x)}(Q_i)}^{q_i^{+}}, \|v\|_{L^{q(x)}(Q_i)}^{q_i^{-}}\}\\
&\leq C\|v\|_{W^{s(x,y),p(x,y)}(Q_i)}^{q_i^{-}} \\
&\leq C\varrho_{W^{s(x,y),p(x,y)}(Q_i)}(v)^{q_i^{-}/p_i^{+}} \\
&\leq C\varrho_{W^{s(x,y),p(x,y)}(Q_i)}(v).
\end{align*}
Summing up for all $i$, we obtain 
\begin{equation*} \label{E}
\varrho_{L^{q(x)}(\Omega)}(v)\leq C\varrho_{W^{s(x,y),p(x,y)}(\Omega)}(v) <\infty.
\end{equation*}
Note that we can take $C$ in \eqref{E} depending only on $N$, $p^{-}$, $p^{+}$, $s^{-}$, $s^{+}$ and $C_L$. Such a constant does not depend on $\varepsilon$. $C_L$ depends only on $p^{-}$, $p^{+}$, $c_p$ and $c_s$. \par
Finally, we consider the general case without the condition (E1). Take any $v\in C_c^\infty(\Omega)$. Let $\{q_n\}$ be a sequence satisfying (E1) with $\varepsilon=\varepsilon_n=o_n(1)$ and approximating $q$ pointwise from below. By Fatou's lemma (more precisely by Corollary 3.5.4 in \cite{Diening}), $\displaystyle \|v\|_{q(x)}\leq\liminf_{n\to\infty}\|v\|_{q_n(x)}$. Since we can take $C$ in \eqref{E} independent of $\varepsilon$, we get $\|v\|_{q(x)}\leq C \|v\|_{W^{s(x,y),p(x,y)}}$. We can conclude by a density argument. 
\end{proof}
Furthermore, the continuous embedding becomes compact under a certain logarithmic condition on the approach speed of $q(\cdot)$ to the critical Sobolev exponent. It is noteworthy that the compact embedding holds even if $q(\cdot)$ reaches the critical Sobolev exponent. 

\CriticalCompactEmbedding*

\begin{proof}
Without loss of generality, we may assume $x_0=0\in\Omega$. 
If a measurable set $A\subset\Omega$ satisfies $|A|<1$, then we have
\begin{align*}
\int_{A\cap \{v\geq 1\}}|v|^{q(x)}dx&\leq \|v\|_{L^{q|_A^{+}}(A)}^{q|_A^{+}}\\ 
&\leq C_1\|v\|_{L^{p_s^*(x)}(A)}^{q|_A^{+}}\|1\|_{L^{\frac{p_s^*(x)q|_A^{+}}{p_s^*(x)-q|_A^{+}}}(A)}^{q|_A^{+}} \\
&\leq C_2 \|v\|_{W^{s(x,y),p(x,y)}(A)}^{q|_A^{+}}\max\{|A|^{\frac{p_s^*|_A^{+}-q|_A^{+}}{p_s^*|_A^{+}}}, |A|^{\frac{p_s^*|_A^{-}-q|_A^{+}}{p_s^*|_A^{-}}}\} \\
&\leq C_3 \|v\|_{W^{s(x,y),p(x,y)}(A)}^{q|_A^{+}} |A|^{\frac{p_s^*|_A^{-}-q|_A^{+}}{p_s^*|_A^{-}}}.
\end{align*}
From \eqref{CriticalRate} and the log-H\"{o}lder continuity of $p$ and $s$, we have
\begin{align*}
&\phantom{=}\left.p_s^*\right|_{B_{\varepsilon^n}(0)\setminus B_{\varepsilon^{n+1}}(0)}^{-}-\left. q\right|_{B_{\varepsilon^n}(0)\setminus B_{\varepsilon^{n+1}}(0)}^{+} \\
&\leq \frac{C_0}{(-\log \varepsilon^{n+1})^\beta}+\frac{C_0'}{-\log(\varepsilon^n-\varepsilon^{n+1})} \\
&\leq \frac{C}{(-\log \varepsilon^{n+1})^\beta}
\end{align*}
for some constants $C_0',C>0$. Note that
\begin{align*}
|B_{\varepsilon^n}(0)\setminus B_{\varepsilon^{n+1}}(0)|^{\frac{C}{(p_s^*)^{-}(-\log\varepsilon^{n+1})^\beta}}&\leq \left(|B_1(0)|(\varepsilon^{Nn}-\varepsilon^{N(n+1)})\right)^{\frac{C}{(p_s^*)^{-}(n+1)^\beta (-\log\varepsilon)^\beta}} \\
&\leq \max\{|B_1(0)|^{\frac{C}{(p_s^*)^{-}\cdot 2^\beta (\log 2)^\beta}},1\}\cdot \varepsilon^{\frac{n^\beta}{(n+1)^\beta}\cdot\frac{CN}{(p_s^*)^{-} (-\log\varepsilon)^\beta}n^{1-\beta}} \\
&= O(\delta_\varepsilon^{n^{1-\beta}})\quad\text{as $n\to\infty$}
\end{align*}
for $\varepsilon\in (0,1/2)$, where $\delta_\varepsilon$ is a constant depending on $\varepsilon$ such that $\delta_\varepsilon\to 0$ ($\varepsilon\to +0$). \par
Let $\rho>0$. Take any $v$ such that $\|v\|_{W^{s(x,y),p(x,y)}(\Omega)}\leq \rho$. 
Applying Lebesgue's convergence theorem to $n\mapsto \delta^{n^{1-\beta}}$ and the counting measure, we get
\begin{align*}
\int_{B_\varepsilon(0)}|v|^{q(x)}dx &= \int_{B_\varepsilon(0)\cap\{v<1\}}|v|^{q(x)}dx + \int_{B_\varepsilon(0)\cap\{v\geq 1\}}|v|^{q(x)}dx \\
&\leq |B_\varepsilon(0)|+C_3\sum_{n=1}^\infty |B_{\varepsilon^n}(0)\setminus B_{\varepsilon^{n+1}}(0)|^{\frac{\left.p_s^*\right|_{B_{\varepsilon^n}(0)\setminus B_{\varepsilon^{n+1}}(0)}^{-}-\left. q\right|_{B_{\varepsilon^n}(0)\setminus B_{\varepsilon^{n+1}}(0)}^{+}}{\left. p_s^*\right|_{B_{\varepsilon^n}(0)\setminus B_{\varepsilon^{n+1}}(0)}^{-}}} \\
&\leq |B_\varepsilon(0)|+C_3\sum_{n=1}^\infty |B_{\varepsilon^n}(0)\setminus B_{\varepsilon^{n+1}}(0)|^{\frac{C}{(p_s^*)^{-}(\log(1/\varepsilon^{n+1}))^\beta}} \\
&\leq |B_\varepsilon(0)|+C_4\sum_{n=1}^\infty \delta_\varepsilon^{n^{1-\beta}} \\
&\to 0\quad (\varepsilon\to +0).
\end{align*}
Here, $C_4$ depends only on the upper bound $\rho$ of $\|v\|_{W^{s(x,y),p(x,y)}(\Omega)}$ and is independent of $v$. Therefore, 
\begin{equation}\label{errorVanishing}
\lim_{\varepsilon\to +0}\sup_{\|v\|_{W^{s(x,y),p(x,y)}(\Omega)}\leq \rho}\int_{B_\varepsilon(0)}|v|^{q(x)}dx=0.
\end{equation}
Assume $v_n\rightharpoonup v$ in $W^{s(x,y),p(x,y)}(\Omega)$ weakly. By a classical compact embedding, passing to a subsequence, $v_n\to v$ a.e. For any $\varepsilon>0$, by the compact embedding $W^{s(x,y),p(x,y)}(\Omega\setminus B_\varepsilon (0)) \hookrightarrow L^{q(x)}(\Omega\setminus B_\varepsilon (0))$, we have
\[
\int_{\Omega \setminus B_\varepsilon (0)}|v_n-v|^{q(x)}dx\to 0
\]
Combining this with \eqref{errorVanishing}, we obtain
\[
\int_{\Omega}|v_n-v|^{q(x)}dx\to 0,
\]
that is, $v_n\to v$ in $L^{q(x)}(\Omega)$. 
\end{proof}
Since the continuous embedding $W_{0}^{s(x,y),p(x,y)}(\Omega)\hookrightarrow W^{s(x,y),p(x,y)}(\Omega)$ holds obviously, we also obtain the compact embedding $W_{0}^{s(x,y),p(x,y)}(\Omega)\hookrightarrow L^{q(x)}(\Omega)$. \par
The logarithmic condition on the approach speed in Theorem \ref{CriticalCompactEmbedding} is sharp in the following sense.

\CriticalNotCompactEmbedding*

\begin{proof}
By translation and scaling, without loss of generality, we may assume $x_0=0\in B_1(0)\subset \Omega$. Let $\varphi\in C_c^\infty(\mathbb{R}^N)$ be such that $\chi_{B_{1/2}(0)}\leq\varphi\leq\chi_{B_1(0)}$ and define $\varphi_n(x)\coloneqq n^{\frac{N-p(x)s(x)}{p(x)}}\varphi(nx)$, where $p(x)\coloneqq p(x,x),s(x)\coloneqq s(x,x)$. Given arbitrary continuous extensions of $p(\cdot,\cdot)$ and $s(\cdot,\cdot)$ to $\mathbb{R}^N\times\mathbb{R}^N$ that do not change the maxima and the minima, we have
\begin{align*}
&\phantom{=}\int_{\Omega} \int_{\Omega}\frac{|\varphi_n(x)-\varphi_n(y)|^{p(x,y)}}{|x-y|^{N+p(x,y)s(x,y)}}dxdy \\
&\leq \int_{\mathbb{R}^N} \int_{\mathbb{R}^N}\frac{|\varphi_n(x)-\varphi_n(y)|^{p(x,y)}}{|x-y|^{N+p(x,y)s(x,y)}}dxdy \\
&= \int_{\mathbb{R}^N} \int_{\mathbb{R}^N}\frac{n^{\frac{N-p(x'/n)s(x'/n)}{p(x'/n)}p(x'/n,y'/n)} |\varphi(x')-n^{\frac{N-p(x'/n)s(x'/n)}{p(x'/n)}-\frac{N-p(y'/n)s(y'/n)}{p(y'/n)}} \varphi(y')|^{p(x'/n,y'/n)}}{|x'/n-y'/n|^{N+p(x'/n,y'/n)s(x'/n,y'/n)}}\cdot \frac{1}{n^{2N}}dx'dy' \\
&= \int_{\mathbb{R}^N} \int_{\mathbb{R}^N} e^{N\left(\frac{p(x'/n,y'/n)}{p(x'/n)}-1\right)\log n +\left(p(x'/n,y'/n)s(x'/n,y'/n)-\frac{p(x'/n,y'/n)}{p(x'/n)}p(x'/n)s(x'/n)\right)\log n} \\
&\phantom{=======}\times\frac{|\varphi(x')-n^{\frac{N-p(x'/n)s(x'/n)}{p(x'/n)}-\frac{N-p(y'/n)s(y'/n)}{p(y'/n)}} \varphi(y')|^{p(x'/n,y'/n)}}{|x'-y'|^{N+p(x'/n,y'/n)s(x'/n,y'/n)}}dx'dy' \\
&\to \int_{\mathbb{R}^N} \int_{\mathbb{R}^N}\frac{|\varphi(x')-\varphi(y')|^{p(0)}}{|x'-y'|^{N+p(0)s(0)}}dx'dy'<\infty\quad(n\to\infty).
\end{align*}
Here, we used Lebesgue's convergence theorem and the fact that from the H\"{o}lder continuity of $p(\cdot,\cdot)$ and $s(\cdot,\cdot)$, we have
\begin{align*}
&\phantom{=}n^{\frac{N-p(x'/n)s(x'/n)}{p(x'/n)}-\frac{N-p(y'/n)s(y'/n)}{p(y'/n)}} \\
&= e^{\left(\frac{N-p(x'/n)s(x'/n)}{p(x'/n)}-\frac{N-p(y'/n)s(y'/n)}{p(y'/n)}\right)\log n} \\
&\to 1\quad (n\to\infty)
\end{align*}
and
\[
N\left(\frac{p(x'/n,y'/n)}{p(x'/n)}-1\right)\log n +\left(p(x'/n,y'/n)s(x'/n,y'/n)-\frac{p(x'/n,y'/n)}{p(x'/n)}p(x'/n)s(x'/n)\right)\log n\to 0
\]
as $n\to \infty$. Therefore, $\{\varphi_n\}$ is bounded in $W^{s(x,y),p(x,y)}(\Omega)$ and hence has a weak limit $\psi$ up to a subsequence. On the other hand, since $\|\varphi_n\|_{p^{-}}\to 0$ as $n\to\infty$, we have $\psi=0$ a.e. 
However, for $n$ sufficiently large,
\begin{align*}
\int_{\Omega}|\varphi_n(x)|^{q(x)}dx &=\int_{B_1(0)} n^{\frac{N-p(x'/n)s(x'/n)}{p(x'/n)}q(x'/n)-N} |\varphi(x')|^{q(x'/n)}dx' \\
&=\int_{B_1(0)} n^{-\frac{N-p(x'/n)s(x'/n)}{p(x'/n)}(p_s^*(x'/n)-q(x'/n))} |\varphi(x')|^{q(x'/n)}dx' \\
&\geq \int_{B_{1/2}(0)} n^{-\frac{N-p(x'/n)s(x'/n)}{p(x'/n)}(p_s^*(x'/n)-q(x'/n))} dx' \\
&\geq \int_{B_{1/2}(0)} n^{-\frac{N-p(x'/n)s(x'/n)}{p(x'/n)}\cdot\frac{C_0}{-\log(|x'/n|)}} dx' \\
&\geq \int_{B_{1/2}(0)\setminus B_{1/4}(0)} \exp\left(-\frac{N-p(x'/n)s(x'/n)}{p(x'/n)}\cdot\frac{C_0}{(1/2)\log n}\log n\right)dx' \\
&\geq \exp\left(-\frac{N-p^{-}s^{-}}{p^{-}}\cdot 2C_0\right)\cdot |B_{1/2}(0)\setminus B_{1/4}(0)| \\
&>0.
\end{align*}
Therefore, $\{\varphi_n\}$ does not have any subsequence converging to $0$ strongly in $L^{q(x)}(\Omega)$. 
\end{proof}

\section{Proof of the main theorem}
The following functional $I$ associated with the equation \eqref{cho} is introduced for discussion using the variational method:
\[
I[u]\coloneqq \int_{\Omega} \int_{\Omega}\frac{|u(x)-u(y)|^{p(x,y)}}{p(x,y)|x-y|^{N+s(x,y)p(x,y)}}dxdy +\int_{\Omega}\frac{|u|^{p(x,x)}}{p(x,x)}dx - \int_{\Omega} \int_{\Omega}\frac{|u(x)|^{r(x)}|u(y)|^{r(y)}}{2r(x)|x-y|^{\frac{\alpha(x)+\alpha(y)}{2}}}dxdy.
\]
A critical point of $I: W^{s(x,y),p(x,y)}_0(\Omega)\to\mathbb{R}$ is called a \textit{weak solution} of the equation \eqref{cho}. In the case of variable order, the concept of classical solutions no longer makes sense. Therefore, a weak solution is often simply called a solution. \par

To obtain a nontrivial weak solution of \eqref{cho}, we use the classical mountain pass lemma. 
Now, let us prove our main theorem. 
\begin{proof}
First, we check that $I$ has a mountain pass geometry. That is, we prove the following holds:
\begin{itemize}
\item[(MP)] There exist positive constants $d$ and $\rho$ such that $I[u]\geq d$ if $\|u\|=\rho$, and there exists $v\in W^{s(x,y),p(x,y)}_0(\Omega)$ with $\|v\|>\rho$ such that $I[v]\leq 0= I[0]$. 
\end{itemize}
By the Hardy-Littlewood-Sobolev inequality, we have
\begin{align*}
I[u] &\geq C_1 \varrho(u)-C_2\| |u|^{r(x)}\|_{\sigma_\alpha^{+}}^2-C_3 \| |u|^{r(x)}\|_{\sigma_\alpha^{-}}^2 \\
&\geq C_4\min\{\|u\|^{p^{+}},\|u\|^{p^{-}}\}-C_5(\| u\|_{r(x)\sigma_\alpha^{+}}^{2r^{+}}+ \| u\|_{r(x)\sigma_\alpha^{+}}^{2r^{-}}+ \| u\|_{r(x)\sigma_\alpha^{-}}^{2r^{+}}+ \| u\|_{r(x)\sigma_\alpha^{-}}^{2r^{-}}) \\
&\geq C_4\min\{\|u\|^{p^{+}},\|u\|^{p^{-}}\}-C_6(\|u\|^{2r^{+}}+\|u\|^{2r^{-}}).
\end{align*}
Since $2r^{+}\geq 2r^{-}>p^{+}\geq p^{-}$, for $\rho>0$ sufficiently small, there exists $d>0$ such that $I[u]\geq d$ if $\|u\|=\rho$.\par
On the other hand, for fixed $u\in W^{s(x,y),p(x,y)}_0(\Omega)\setminus\{0\}$ and $t>1$, we have
\begin{align*}
I[tu] &\leq C_1' t^{p^{+}} \max\{\|u\|^{p^{+}},\|u\|^{p^{-}}\}-C_2' t^{2r^{-}} \int_{\Omega} \int_{\Omega}\frac{|u(x)|^{r(x)}|u(y)|^{r(y)}}{2r(x)|x-y|^{\frac{\alpha(x)+\alpha(y)}{2}}}dxdy.
\end{align*}
Hence, for $t>1$ sufficiently large, we have $I[tu] <0$. \par
Next, we check the (PS) condition. Let $\{u_n\}$ be a sequence such that $\{I[u_n]\}$ is bounded and $I'[u_n]\to 0$ ($n\to\infty$). For $\beta\in (1/(2r^{-}),1/p^{+})$, we have 
\begin{align*}
I[u_n]-\beta I'[u_n](u_n)&\geq \left(\frac{1}{p^{+}}-\beta\right)\varrho(u)+\left(\beta-\frac{1}{2r^{-}}\right) \int_{\Omega} \int_{\Omega}\frac{|u(x)|^{r(x)}|u(y)|^{r(y)}}{|x-y|^{\frac{\alpha(x)+\alpha(y)}{2}}}dxdy \\
&\geq \left(\frac{1}{p^{+}}-\beta\right) \min\{\|u\|^{p^{+}},\|u\|^{p^{-}}\}.
\end{align*}
Hence, $\{u_n\}$ is bounded. We may assume $u_n\rightharpoonup u$ weakly in $W^{s(x,y),p(x,y)}_0(\Omega)$ up to a subsequence. 
By the compact embeddings to $L^{r(x)\sigma_\alpha^{+}}(\Omega)$ and $L^{r(x)\sigma_\alpha^{-}}(\Omega)$, passing to a subsequence, $u_n\to u$ in $L^{r(x)\sigma_\alpha^{+}}(\Omega)\cap L^{r(x)\sigma_\alpha^{-}}(\Omega)$. By the Hardy-Littlewood-Sobolev inequality, 
\[
v\mapsto \int_{\Omega} \int_{\Omega}\frac{|v(x)|^{r(x)}|v(y)|^{r(y)}}{|x-y|^{\frac{\alpha(x)+\alpha(y)}{2}}}dxdy
\]
is strongly continuous in $L^{r(x)\sigma_\alpha^{+}}(\Omega)\cap L^{r(x)\sigma_\alpha^{-}}(\Omega)$. Hence,
\[
 \int_{\Omega} \int_{\Omega}\frac{|u_n(x)|^{r(x)}|u_n(y)|^{r(y)}}{|x-y|^{\frac{\alpha(x)+\alpha(y)}{2}}}dxdy\to  \int_{\Omega} \int_{\Omega}\frac{|u(x)|^{r(x)}|u(y)|^{r(y)}}{|x-y|^{\frac{\alpha(x)+\alpha(y)}{2}}}dxdy
\]
as $n\to\infty$. Combining this with $I'[u_n](u_n)\to 0$, we obtain $\varrho(u_n)\to\varrho(u)$. This also implies $\|u_n\|\to\|u\|$. By the uniform convexity of $W^{s(x,y),p(x,y)}(\Omega)$, we conclude $u_n\to u$ strongly. \par
From the above, we can apply the mountain pass lemma and obtain a nontrivial weak solution. 
\end{proof}

\end{document}